%% file: main.tex
\documentclass[pdflatex,sn-mathphys-num]{sn-jnl}
\usepackage{graphicx}
\usepackage{multirow}
\usepackage{amsmath,amssymb,amsfonts}
\usepackage{amsthm}
\usepackage{mathrsfs}
\usepackage[title]{appendix}
\usepackage{xcolor}
\usepackage{textcomp}
\usepackage{manyfoot}
\usepackage{booktabs}
\usepackage{algorithm}
\usepackage{algorithmicx}
\usepackage{algpseudocode}
\usepackage{listings}
\usepackage{makecell}
\usepackage[all,cmtip]{xy}
\usepackage{mathtools}
\usepackage{microtype}
\usepackage{latexsym}
\usepackage{yfonts}
\usepackage{graphics}
\usepackage{array}
\usepackage{enumitem}
\usepackage{float}
\input{init.tex}

\begin{document}

\title[A penalised Saito functional for heuristic search of free line arrangements]{A penalised Saito functional for heuristic search of free line arrangements}

\author*{\fnm{Tom\'as} \sur{S. R. Silva}}\email{tomas@ime.unicamp.br}

\affil{\orgdiv{Instituto de Matemática, Estatística e Computação Científica (IMECC)}, \orgname{Universidade Estadual de Campinas (UNICAMP)}, \orgaddress{ \postcode{13083-859}, \city{Campinas}, \state{S\~ao Paulo}, \country{Brazil}}}

\abstract{
\unboldmath
We introduce the penalised Saito functional $\mathfrak S_{\lambda,\beta}(\mathcal{A};d_1,d_2)$ for a reduced arrangement $\mathcal{A}$ of $n$ lines and a prescribed pair $d_1+d_2=n-1$. It measures the alignment of a candidate Saito determinant with the defining polynomial while penalising the failure of the candidate derivations to be logarithmic. We prove that the functional takes values in $[0,1]$, vanishes exactly when $\mathcal{A}$ is free with exponents $(1,d_1,d_2)$, and lies strictly between $0$ and $1$ otherwise. For fixed $(d_1,d_2)$, it is upper semicontinuous on the reduced configuration space, continuous at arrangements free with the prescribed pair, and converges as $\lambda\to\infty$ to the corresponding binary freeness test. 

We use a numerical approximation of this functional, together with a small $b_2$-shell term, to guide fixed-cardinality line-replacement searches over $\mathbb{Q}$ and selected quadratic extensions. Numerical values are used only to select candidates; every reported arrangement is certified in exact arithmetic using Saito's criterion. At the current snapshot, the certified database contains $6{,}146$ representatives with distinct Weisfeiler--Leman fingerprints and cardinalities up to $n=28$. Among them, $3{,}012$ have multiplicity gap $\epsilon(\mathcal{A})=d_1-m(\mathcal{A})\geq2$, including lower-bound-extremal examples with $\epsilon=7$. These non-supersolvable arrangements provide test cases for studying realisation spaces and the persistence of freeness among realisations of the same intersection lattice, in connection with Terao's conjecture.
}



\maketitle

\newpage

\section{Introduction}
\label{sec:introduction}

Hyperplane arrangements lie at a rich intersection of algebraic geometry, combinatorics, and singularity theory. Among the many properties an arrangement may possess, \emph{freeness} occupies a central role. Introduced by Saito~\cite{Saito1980}, freeness is defined in terms of the module of logarithmic derivations associated with the defining polynomial of the arrangement: an arrangement is free when this module splits as a direct sum of graded components of prescribed degrees. This splitting condition is closely related to the existence of polynomial vector fields satisfying divisibility constraints along the arrangement; see, for instance,~\cite{Orlik1992,Dimca2017,Yoshinaga2014,Jardim2024}.

Despite the simplicity of its definition, freeness remains difficult to characterise, detect, and construct systematically. For arrangements with a fixed intersection lattice, the locus of free realisations is Zariski open in the corresponding realisation space~\cite{Yuzvinsky1993}. Terao's conjecture~\cite{Terao80} predicts that this locus is either empty or the entire realisation space. Indeed, a nonempty proper free locus would yield two arrangements with the same intersection lattice but different freeness behaviour, and hence a counterexample to the conjecture. Even for line arrangements in $\mathbb P^2$, Terao's conjecture remains widely open, and much recent work has focused on clarifying the relation between combinatorial data and freeness-related algebraic invariants, such as the module of logarithmic derivations and the minimal degree of Jacobian relations
\cite{Schenck2012,Dimca17a,Dimca2017b,DIMCA2016,Marchesi_2023,abe2025additiontheoremszieglerpairs,Dimca2026freelinearrangementslow,kuhne2024numericalteraosconjectureziegler,Kuhne_2024,pokora2024construction,Pokora_2024}.

From both the computational and theoretical viewpoints, constructing line arrangements with prescribed algebraic properties is a difficult problem. The configuration space grows rapidly with the number of lines, while freeness is a highly non-generic property: free arrangements appear comparatively sparsely within the full space of configurations. In practice, explicit examples are usually obtained through geometric insights, addition--deletion arguments, or symbolic computations tailored to specific families; see, for example,~\cite{Dimca2017b,Abe2020,pokora2024construction,Kabat2018,Mu2015}. On the computational side, large-scale investigations of realisable rank-three matroids have illustrated both the usefulness and the limitations of exhaustive enumeration~\cite{Barakat2021}. A search objective tied directly to a necessary-and-sufficient freeness criterion could therefore provide an additional source of examples for the theory.

This is precisely the viewpoint adopted in this paper. Rather than treating freeness only as a property to be checked \emph{after} an arrangement has been constructed, we introduce a bounded functional that can guide the
construction process itself.

The starting point is Saito's criterion
(Theorem~\ref{thm:Saito-criterion}), which requires logarithmic derivations of degrees $d_1$ and $d_2$ whose Saito determinant is a nonzero scalar multiple of the defining polynomial. When $d_1+d_2=n-1$, Lemma~\ref{lem:divisibility-obstruction} shows that proportionality to the defining polynomial is automatic for every exact logarithmic pair; the remaining distinction is whether the determinant vanishes. This suggests relaxing the logarithmicity conditions themselves. We therefore work in the fixed ambient spaces of derivations of the prescribed degrees, measure their failure to be logarithmic by canonical residuals, and combine this penalty with the determinant condition to obtain a graded functional: the \emph{penalised Saito functional} (Definition~\ref{def:saito-functional})
\[
\mathfrak S_{\lambda,\beta}(\mathcal{A};d_1,d_2),
\]
defined in Section~\ref{sec:penalised-saito-functional} for every reduced arrangement $\mathcal{A}$, every prescribed pair $d_1+d_2=n-1$, every $\lambda>0$, and every $0<\beta<1$. The functional compares the candidate Saito determinant with the defining polynomial while simultaneously penalising the failure of the candidate derivations to satisfy
the logarithmic divisibility conditions.

The main properties of the functional are established in Theorem~\ref{thm:penalised-saito}:

\begin{enumerate}
  \item[\textup{(i)}]
  \[
  0\leq
  \mathfrak S_{\lambda,\beta}(\mathcal{A};d_1,d_2)
  \leq1;
  \]

  \item[\textup{(ii)}]
  \[
  \mathfrak S_{\lambda,\beta}(\mathcal{A};d_1,d_2)=0
  \]
  if and only if $\mathcal{A}$ is free with exponents
  $(1,d_1,d_2)$, while an arrangement that is not free with the
  prescribed exponent pair has functional value strictly between $0$
  and $1$;

  \item[\textup{(iii)}]
  for fixed $(d_1,d_2)$, the functional is upper semicontinuous on the
  reduced configuration space and is continuous at arrangements that are
  free with the prescribed exponent pair.
\end{enumerate}

Moreover, the functional is nondecreasing in the penalty parameter $\lambda$, and its limit as $\lambda\to\infty$ recovers the corresponding binary target-pair freeness test. Although its zero locus is intrinsic, its positive values depend on the chosen Hermitian normalisation and on the parameters $\lambda$ and $\beta$. It should therefore be viewed as a bounded, coordinate-normalised search energy, rather than a projective invariant or a metric distance to the free locus.

In recent years, machine learning and heuristic optimisation have been used as exploratory tools in several areas of pure mathematics, particularly in settings involving large combinatorial or geometric search spaces \cite{He2022,He2023,He2025,Davies2021,Berglund2024,Gukov2021,Gukov2025,
heyes2026neuralnumericalmethodsmathrmg2structures, hirst2026minimisingwillmoreenergyneural,Aggarwal2024}. The feature emphasised here is that the principal search signal is a penalised relaxation of Saito's exact algebraic criterion, while exact arithmetic is retained for final certification.

Building on this relaxation, we develop a computational framework that explores the space of line arrangements in $\mathbb P^2$ with fixed cardinality and prescribed target exponents. A search move replaces one line of the current arrangement by another,
\[
\mathcal{A}'
=
\bigl(\mathcal{A}\setminus\{L_-\}\bigr)\cup\{L_+\},
\]
so that the cardinality remains fixed. Several search strategies are considered, including simulated annealing~\cite{Kirkpatrick1983}, greedy search~\cite{Cuntz2021}, random walks~\cite{Solis1981}, and MAP--Elites~\cite{mouret2015illuminatingsearchspacesmapping, Cully2018-nm}. These methods share a common proposal, numerical-evaluation, and exact-certification pipeline, but differ in how they select parents and propose, accept, or retain candidate arrangements.

This fixed-cardinality line-replacement philosophy is closely related to the greedy method of Cuntz~\cite{Cuntz2021}. Working primarily over finite fields, Cuntz replaces one line by a line through two existing intersection points and selects moves using a combinatorial objective; promising matroids are then studied for realisability in characteristic zero. The principal application is simpliciality. In freeness-related experiments, integral roots of the characteristic polynomial provide a necessary but not sufficient combinatorial filter, and the method recovers known examples that are free but not inductively free or not recursively free. Our framework retains the broad remove-and-replace move but changes the principal objective: the penalised Saito functional is evaluated on the actual realisation and detects freeness with the prescribed exponent pair at value zero, whereas characteristic-polynomial factorisation supplies a necessary combinatorial shell. We retain that shell as an inexpensive auxiliary term and require exact Saito certification for every reported discovery.

For the energy-based engines, the numerical objective combines an
approximation of the raw penalised Saito functional with a small combinatorial term measuring the discrepancy between $b_2(\mathcal{A})$ and the value
\[
b_2^*=(n-1)+d_1d_2
\]
required by Terao's factorisation theorem (Theorem~\ref{thm:Terao-factorisation}) for the target exponents. Separately, the replacement proposal mechanism used across the campaigns is itself biased towards the target $b_2$-shell: exact-shell proposals are favoured, while a controlled amount of off-shell exploration is retained. Thus the $b_2$ condition supplies inexpensive combinatorial guidance at the proposal level and, where applicable, through the auxiliary energy term, whereas exact Saito certification remains the sole proof of freeness. The precise conventions and the general search procedure are described in Section~\ref{sec:penalised-saito-functional}.

The search campaigns are carried out over $\mathbb{Q}$ and selected real and imaginary quadratic extensions. The numerical functional is used only to rank and select promising candidates. Every arrangement reported as free is verified independently in exact arithmetic over its original coefficient field using Saito's criterion. Thus numerical optimisation guides the exploration, whereas exact algebra supplies the final correctness guarantee. The searches focus on non-trivial arrangements and exclude pencils and near-pencils.

Section~\ref{sec:experimental_results} presents certified free arrangements found by the search campaigns, with particular emphasis on the multiplicity gap $\epsilon(\mathcal{A})=d_1-m(\mathcal{A})$. It records their distribution across cardinality ranges and exponent pairs, highlights large-gap arrangements attaining the lower Dimca--K\"uhne--Pokora bound~\cite{Dimca2026freelinearrangementslow}, and discusses their relevance as test cases in the study of Terao's conjecture.

The contribution of this work is threefold. First, the penalised Saito functional provides a bounded, upper semicontinuous relaxation whose zero locus is exactly the locus of arrangements free with the prescribed exponent pair. Second, the computational framework provides a mechanism for generating and exactly certifying free line arrangements across a range of cardinalities, exponent pairs, and coefficient fields. Third, the resulting dataset supplies a large collection of combinatorially diverse, exactly certified free arrangements, including non-supersolvable examples with large multiplicity gap. These arrangements provide test cases for identifying combinatorial patterns associated with freeness and for stress-testing searches for potential counterexamples to Terao's conjecture by investigating whether their intersection lattices also admit non-free realisations.

\medskip

\noindent\textbf{Paper organisation.}
In \S\ref{sec:preliminaries-arrangements} we review the necessary background on line arrangements. In \S\ref{sec:penalised-saito-functional} we introduce the penalised Saito functional, establish its main properties, and describe the general fixed-cardinality search procedure. In \S\ref{sec:experimental_results} we discuss the outputs of the search campaigns. Finally, in \S\ref{sec:conclusion} we summarise the conclusions and discuss possible directions for future work.

\section{Line arrangements in \texorpdfstring{$\mathbb{P}^2$}{P2}}
\label{sec:preliminaries-arrangements}

In this section we recall the basic notions concerning line arrangements, logarithmic derivations, and freeness that will be used throughout the paper. Standard references include~\cite{Orlik1992,Dimca2017,Yoshinaga2014}.

Let $S=\mathbb C[x,y,z]$ be the homogeneous coordinate ring of $\mathbb P^2$. A \emph{reduced projective line arrangement} is a finite collection
\[
\mathcal{A}=\{L_1,\ldots,L_n\}
\]
of $n$ distinct projective lines. For each $i$, choose a linear form
$\alpha_i\in S_1$ such that $L_i=V(\alpha_i)$, and associate with
$\mathcal{A}$ the reduced defining polynomial
\footnote{When no confusion can arise, we write $Q$ in place of
$Q(\mathcal{A})$.}
\[
Q(\mathcal{A})=\prod_{i=1}^n\alpha_i.
\]
This polynomial is determined by the arrangement up to multiplication by a nonzero scalar. Thus $\mathcal{A}$ may be identified with the reduced divisor $V(Q(\mathcal{A}))\subset\mathbb P^2$. It is often convenient to work in $\mathbb C^3$, where the same polynomial defines a central hyperplane arrangement. We pass freely between these two viewpoints.

The arrangement is called \emph{essential} if the corresponding central hyperplanes intersect only at the origin, or equivalently if the projective lines do not all pass through a common point. Thus a nonessential reduced arrangement of at least two lines is a \emph{pencil}.

Let
\[
\mathcal{A}^{\mathrm{cen}}
=
\{H_1,\ldots,H_n\},
\qquad
H_i=V_{\mathbb C^3}(\alpha_i),
\]
be the central arrangement associated with $\mathcal{A}$. Its
\emph{intersection lattice} is
\[
\mathcal L(\mathcal{A})
=
\left\{
\bigcap_{i\in I}H_i:
I\subseteq\{1,\ldots,n\}
\right\},
\]
partially ordered by reverse inclusion, where the intersection corresponding to $I=\varnothing$ is $\mathbb C^3$. Thus the minimal element is $\hat 0=\mathbb C^3$. The rank of a flat $X\in\mathcal L(\mathcal{A})$ is its codimension in $\mathbb C^3$, and we write
\[
\mathcal{L}_2(\mathcal{A})
=
\left\{
X\in\mathcal L(\mathcal{A}):
\operatorname{codim}_{\mathbb C^3}X=2
\right\}.
\]
If $\mathcal{A}$ is essential, then $\{0\}$ is its unique rank-three flat; if $\mathcal{A}$ is a pencil, its intersection lattice has rank two.

The rank-one flats are the hyperplanes $H_i$, corresponding to the projective lines $L_i$, while the rank-two flats correspond to the intersection points of the projective arrangement. Thus $\mathcal L(\mathcal{A})$ records the line--point incidence structure of $\mathcal{A}$. For $X\in\mathcal{L}_2(\mathcal{A})$, define its multiplicity by
\[
m_X
=
\#\left\{
i\in\{1,\ldots,n\}:X\subseteq H_i
\right\}.
\]
If $X$ corresponds to the projective point $p\in\mathbb P^2$, we also write $m_p=m_X$. Equivalently,
\[
m_p
=
\#\{L_i\in\mathcal{A}:p\in L_i\}.
\]
We write $t_m(\mathcal{A})$ for the number of projective intersection points of multiplicity $m$. These numbers satisfy the double-counting identity
\begin{equation}
\label{eq:double-counting}
\sum_{m\geq2}\binom{m}{2}t_m(\mathcal{A})
=
\sum_{X\in\mathcal{L}_2(\mathcal{A})}
\binom{m_X}{2}
=
\binom{n}{2},
\end{equation}
since every pair of distinct lines meets in exactly one projective point.

\subsection{Characteristic polynomial and Betti numbers}

Associated with $\mathcal L(\mathcal{A})$ is the \emph{M\"obius function}
$\mu: \mathcal{L}(\mathcal{A})\rightarrow \mathbb{Z}$, defined recursively by $\mu(\hat0)=1$ and
\[
\mu(X)=-\sum_{Y<X}\mu(Y)
\]
for $X\neq\hat0$, where $\hat0$ denotes the ambient space. An explicit calculation gives $\mu(H_i)=-1$ for every rank-one flat $H_i$, and $\mu(X)=m_X-1$ for every $X$ in $\mathcal{L}_2(\mathcal{A})$.

The characteristic polynomial of the associated central arrangement is
\[
\chi(\mathcal{A},t)=\sum_{X\in\mathcal L(\mathcal{A})}
\mu(X)t^{\dim X}.
\]
For an arrangement of $n$ projective lines, it has the form
\begin{equation}
\label{eq:chi}
\chi(\mathcal{A},t)
=t^3-nt^2+b_2(\mathcal{A})t-\bigl(b_2(\mathcal{A})-n+1\bigr),
\end{equation}
where
\begin{equation}
\label{eq:b2}
b_2(\mathcal{A})
=\sum_p(m_p-1)
=\sum_{m\geq2}(m-1)t_m(\mathcal{A}).
\end{equation}
Here $b_2(\mathcal{A})$ is the second characteristic coefficient,
equivalently the second Betti number of the central arrangement complement.
Since $\chi(\mathcal{A},1)=0$, one has
\begin{equation}
\label{eq:chi-factored}
\chi(\mathcal{A},t)
=(t-1)\left(t^2-(n-1)t+b_2(\mathcal{A})-n+1\right).
\end{equation}

To distinguish the central and projective complements, put
\[
M_{\mathrm{cen}}(\mathcal{A})
=\mathbb C^3\setminus\bigcup_{L\in\mathcal{A}}H_L,
\qquad
M_{\mathrm{proj}}(\mathcal{A})
=\mathbb P^2\setminus\bigcup_{L\in\mathcal{A}}L,
\]
where $H_L$ is the central hyperplane corresponding to $L$. Their Poincar\'e
polynomials are
\begin{align}
\pi_{\mathrm{cen}}(\mathcal{A},t)
&=(-t)^3\chi(\mathcal{A},-1/t) \notag\\
&=1+nt+b_2(\mathcal{A})t^2
+\bigl(b_2(\mathcal{A})-n+1\bigr)t^3,
\label{eq:poincare-central}
\\
\pi_{\mathrm{proj}}(\mathcal{A},t)
&=\frac{\pi_{\mathrm{cen}}(\mathcal{A},t)}{1+t}
=1+(n-1)t+\bigl(b_2(\mathcal{A})-n+1\bigr)t^2.
\label{eq:poincare-projective}
\end{align}
Thus the quantity in \eqref{eq:b2}, which is used below as a
characteristic-polynomial coefficient, should not be confused with the second
Betti number $b_2(\mathcal{A})-n+1$ of the projective complement.

\subsection{Logarithmic derivations}

Let
\[
\operatorname{Der}_{\mathbb C}(S)
=S\partial_x\oplus S\partial_y\oplus S\partial_z
\]
be the $S$-module of $\mathbb C$-derivations of $S$. A homogeneous derivation
$\theta=f\partial_x+g\partial_y+h\partial_z$ has polynomial degree $d$ if
$f,g,h\in S_d$. The module of \emph{logarithmic derivations} of $\mathcal{A}$ is
\[
D(\mathcal{A})
=\bigl\{\theta\in\operatorname{Der}_{\mathbb C}(S):
\theta(\alpha_i)\in S\alpha_i\text{ for all }i\bigr\}.
\]
Equivalently,
\[
D(\mathcal{A})
=\{\theta\in\operatorname{Der}_{\mathbb C}(S):
\theta(Q)\in SQ\}.
\]
Its elements are the polynomial vector fields tangent to every line of the arrangement. We write $D(\mathcal{A})_d$ for the finite-dimensional space of homogeneous logarithmic derivations of polynomial degree $d$.

A distinguished element is the Euler derivation
\[
\theta_E=x\partial_x+y\partial_y+z\partial_z,
\]
which belongs to $D(\mathcal{A})_1$ because $\theta_E(Q)=nQ$. The module $D(\mathcal{A})$ is a graded reflexive $S$-module of rank $3$.

Logarithmic derivations may also be interpreted using syzygies among the partial derivatives of $Q$. Write
\[
\operatorname{AR}(Q)
:=\{(f,g,h)\in S^3:fQ_x+gQ_y+hQ_z=0\}.
\]
Euler's identity gives the graded decomposition
\[
D(\mathcal{A})\cong S\theta_E\oplus\operatorname{AR}(Q).
\]
Thus freeness amounts to finding sufficiently many independent homogeneous relations among $Q_x,Q_y,Q_z$.

\subsection{Free line arrangements and Saito's criterion}

An arrangement $\mathcal{A}$ is \emph{free} if $D(\mathcal{A})$ is a free
$S$-module. For a free line arrangement one has
\[
D(\mathcal{A})
\cong S(-1)\oplus S(-d_1)\oplus S(-d_2)
\]
for integers $0\leq d_1\leq d_2$, where $S(-1)$ corresponds to the Euler
derivation. The integers $(1,d_1,d_2)$ are called the \emph{exponents} of $\mathcal{A}$; with $d_1\leq d_2$, we write
\[
\exp(\mathcal{A})=(1,d_1,d_2).
\] 
We often call $(d_1,d_2)$ the non-Euler exponent pair. If the arrangement is essential,
then $d_1\geq1$. The case $d_1=0$ accounts for nonessential arrangements
such as pencils.

\begin{theorem}[Saito's criterion~\cite{Saito1980}]
\label{thm:Saito-criterion}
Let $\theta_0,\theta_1,\theta_2\in D(\mathcal{A})$ be homogeneous
derivations, and write
$\theta_i=f_{i1}\partial_x+f_{i2}\partial_y+f_{i3}\partial_z$. Form
\[
M(\theta_0,\theta_1,\theta_2)
=
\begin{pmatrix}
f_{01}&f_{02}&f_{03}\\
f_{11}&f_{12}&f_{13}\\
f_{21}&f_{22}&f_{23}
\end{pmatrix}.
\]
Then $\mathcal{A}$ is free with basis
$\{\theta_0,\theta_1,\theta_2\}$ if and only if
\[
\det M(\theta_0,\theta_1,\theta_2)=cQ
\]
for some $c\in\mathbb C^\times$.
\end{theorem}

For line arrangements, take $\theta_0=\theta_E$. For homogeneous
derivations $u$ and $v$, define
\begin{equation}
\label{eq:B-map}
B(u,v):=\det M(\theta_E,u,v).
\end{equation}
Then $\mathcal{A}$ is free with exponents $(1,d_1,d_2)$ if and only if there
exist $u\in D(\mathcal{A})_{d_1}$ and $v\in D(\mathcal{A})_{d_2}$ such that
\begin{equation}
\label{eq:Saito-det}
B(u,v)=cQ
\end{equation}
for some $c\in\mathbb C^\times$.

The following elementary consequence is central to the construction in Section~\ref{sec:penalised-saito-functional}. It is the degree-specific form of the standard determinant divisibility property for logarithmic derivations; see \cite[Proposition~4.12]{Orlik1992}.

\begin{lemma}
\label{lem:divisibility-obstruction}
Suppose that $u\in D(\mathcal{A})_{d_1}$ and $v\in D(\mathcal{A})_{d_2}$, where $d_1+d_2=n-1$. Then
\[
B(u,v)=cQ
\]
for some $c\in\mathbb C$, possibly $c=0$.
\end{lemma}

\begin{proof}[Proof.]
Write $\alpha_i=a_{i0}x+a_{i1}y+a_{i2}z$. Modulo $\alpha_i$, logarithmicity gives
\[
M(\theta_E,u,v)
\begin{pmatrix}
a_{i0}\\a_{i1}\\a_{i2}
\end{pmatrix}
=
\begin{pmatrix}
\alpha_i\\u(\alpha_i)\\v(\alpha_i)
\end{pmatrix}
\equiv0\pmod{\alpha_i}.
\]
Hence $\alpha_i\mid B(u,v)$ for every $i$. Since the arrangement is reduced, the linear forms $\alpha_i$ are pairwise relatively prime, and therefore $Q\mid B(u,v)$. If $B(u,v)\neq0$, then
\[
\deg B(u,v)=1+d_1+d_2=n=\deg Q,
\]
so the quotient is constant. The assertion is immediate when $B(u,v)=0$.
\end{proof}

Consequently,
\begin{equation}
\label{eq:exact-determinant-line}
B\bigl(D(\mathcal{A})_{d_1},D(\mathcal{A})_{d_2}\bigr)
\subseteq\mathbb C Q
\end{equation}
whenever $d_1+d_2=n-1$. Equivalently, combining Lemma~\ref{lem:divisibility-obstruction} with Saito's criterion gives
\begin{equation}
\label{eq:target-pair-Saito}
\mathcal{A}\text{ is free with exponents }(1,d_1,d_2)
\quad\Longleftrightarrow\quad
B\bigl(D(\mathcal{A})_{d_1},D(\mathcal{A})_{d_2}\bigr)\neq\{0\},
\end{equation}
provided $d_1+d_2=n-1$. In particular, if $\mathcal{A}$ is not free with the prescribed pair, then
\begin{equation}
\label{eq:kernel-determinant-zero}
B\bigl(D(\mathcal{A})_{d_1},D(\mathcal{A})_{d_2}\bigr)=\{0\}.
\end{equation}
Section~\ref{sec:penalised-saito-functional} obtains a graded energy by moving to fixed ambient coefficient spaces and penalising, rather than imposing, logarithmicity.

\subsection{Terao's factorisation and combinatorial constraints}
\label{subsec:numerical-constraints}

A major consequence of freeness is the following factorisation result.

\begin{theorem}[Terao's factorisation theorem~\cite{Terao1981}]
\label{thm:Terao-factorisation}
If $\mathcal{A}$ is free with exponents $(1,d_1,d_2)$, then
\[
\chi(\mathcal{A},t)=(t-1)(t-d_1)(t-d_2).
\]
\end{theorem}

Comparing this factorisation with \eqref{eq:chi-factored} gives the necessary conditions
\begin{equation}
\label{eq:exponent-relations}
d_1+d_2=n-1,
\qquad
d_1d_2=b_2(\mathcal{A})-(n-1).
\end{equation}
For a prescribed pair $d_1+d_2=n-1$, define the \emph{target Betti coefficient}
\begin{equation}
\label{eq:b2-target}
b_2^*(n;d_1,d_2):=(n-1)+d_1d_2.
\end{equation}
Thus freeness with exponents $(1,d_1,d_2)$ implies $b_2(\mathcal{A})=b_2^*(n;d_1,d_2)$, but the converse need not hold.

For a free arrangement, the integers $d_1$ and $d_2$ are the roots of
\[
t^2-(n-1)t+b_2(\mathcal{A})-n+1=0.
\]
For integer exponents to exist, the discriminant
\begin{equation}
\label{eq:discriminant}
\Delta(\mathcal{A})
=(n-1)^2-4\bigl(b_2(\mathcal{A})-n+1\bigr)
\end{equation}
must be a nonnegative perfect square. An arrangement for which this fails cannot be free.

When the discriminant in \eqref{eq:discriminant} is a nonnegative perfect square, the quadratic factor of the characteristic polynomial determines the only possible integral exponent pair $(d_1,d_2)$. This is a necessary arithmetic condition for freeness, but it is not sufficient; freeness must still be verified by Saito's criterion.

In the construction below, $(d_1,d_2)$ denotes a prescribed \emph{target pair}, rather than a pair inferred from the current arrangement. The penalised functional is defined for every reduced arrangement and every
pair $d_1+d_2=n-1$, whether or not
\[
b_2(\mathcal{A})=b_2^*(n;d_1,d_2).
\]
In the search campaigns, the target $b_2$-shell is used in two distinct ways. First, the proposal mechanism favours replacements lying on the shell, while allowing nearby off-shell proposals. Second, for the energy-based engines, a small auxiliary term penalises the discrepancy from the shell. Neither use provides a freeness certificate: every reported arrangement is verified exactly by Saito's criterion.

\section{The penalised Saito functional and search framework}
\label{sec:penalised-saito-functional}

This section introduces a bounded, upper semicontinuous relaxation of Saito's criterion and explains how it is used to guide a fixed-cardinality search for free line arrangements. The construction is formulated over $\mathbb{C}$. Arrangements defined over the exact coefficient fields used in the search are evaluated after choosing an embedding into $\mathbb{C}$ and are certified over their original exact field.

\subsection{The penalised Saito functional}
\label{subsec:saito-functional}
Fix a reduced line arrangement
$\mathcal{A}=\{V(\alpha_1),\ldots,V(\alpha_n)\}$, and fix
$d_1,d_2\in\mathbb Z_{\geq0}$ satisfying $d_1+d_2=n-1$.
Lemma~\ref{lem:divisibility-obstruction} shows that
\[
B\bigl(D(\mathcal{A})_{d_1},D(\mathcal{A})_{d_2}\bigr)
\subseteq
\mathbb C Q(\mathcal{A}).
\]
To obtain a graded energy, we work on fixed ambient coefficient spaces and
penalise the failure of logarithmicity.

Put $S_d=\mathbb C[x,y,z]_d$ and $E_d=S_d^{\oplus3}$. On $S_d$ we use the
Bombieri--Weyl Hermitian inner product
\begin{equation}
\label{eq:BW-inner-product}
\left\langle
\sum_{|\nu|=d}c_\nu x^{\nu_0}y^{\nu_1}z^{\nu_2},
\sum_{|\nu|=d}c'_\nu x^{\nu_0}y^{\nu_1}z^{\nu_2}
\right\rangle_d
=
\sum_{|\nu|=d}
\frac{c_\nu\overline{c'_\nu}}{\binom d\nu},
\end{equation}
where
\[
\binom d\nu
=
\frac{d!}{\nu_0!\nu_1!\nu_2!},
\]
and we give $E_d$ the corresponding direct-sum norm. For each line $L_i=V(\alpha_i)$, rescale its defining form so that
\[
\|\alpha_i\|_1=1.
\]
Throughout the construction, these same unit representatives are used both
in the residual operator $L_{\mathcal{A},d}$ and in the chosen representative
\[
Q(\mathcal{A})
=
\prod_{i=1}^n\alpha_i
\]
of the defining polynomial.

For a unit linear form
\[
\alpha=a_0x+a_1y+a_2z
\]
and $d\geq1$, let
\[
\mu_{\alpha,d}:S_{d-1}\longrightarrow S_d,
\qquad
h\longmapsto\alpha h,
\]
and define
\begin{equation}
\label{eq:multiplication-projector}
\Pi_{\alpha,d}
:=
\mu_{\alpha,d}
(\mu_{\alpha,d}^{*}\mu_{\alpha,d})^{-1}
\mu_{\alpha,d}^{*}.
\end{equation}
Here $\mu_{\alpha,d}^{*}$ is the adjoint with respect to the
Bombieri--Weyl inner products. Multiplication by a nonzero linear form is
injective, so \eqref{eq:multiplication-projector} is well defined and is the
orthogonal projector onto $\alpha S_{d-1}$.

For
\[
u=(f_0,f_1,f_2)\in E_d,
\]
set
\begin{equation}
\label{eq:line-residual}
\rho_{\alpha,d}(u)
:=
(I-\Pi_{\alpha,d})
(a_0f_0+a_1f_1+a_2f_2).
\end{equation}
Define the arrangement residual operator by
\begin{equation}
\label{eq:arrangement-residual}
L_{\mathcal{A},d}:E_d\longrightarrow S_d^{\oplus n},
\qquad
L_{\mathcal{A},d}(u)
:=
\frac{1}{\sqrt n}
\bigl(
\rho_{\alpha_1,d}(u),
\ldots,
\rho_{\alpha_n,d}(u)
\bigr).
\end{equation}
Then
\begin{equation}
\label{eq:residual-kernel}
\ker L_{\mathcal{A},d}
=
D(\mathcal{A})_d,
\end{equation}
because $\rho_{\alpha_i,d}(u)=0$ is equivalent to
\[
u(\alpha_i)\in\alpha_iS_{d-1}
\]
for every $i$. Moreover,
\begin{equation}
\label{eq:residual-energy}
\|L_{\mathcal{A},d}u\|^2
=
\frac{1}{n}
\sum_{i=1}^n
\|\rho_{\alpha_i,d}(u)\|^2.
\end{equation}
Thus the factor $n^{-1/2}$ in~\eqref{eq:arrangement-residual} averages the squared residual over the $n$ lines
and makes its scale more comparable across cardinalities.

For $d=0$, we use the conventions
\[
S_{-1}=0,
\qquad
\mu_{\alpha,0}:0\longrightarrow S_0,
\qquad
\Pi_{\alpha,0}=0.
\]
With these conventions, the same kernel identity remains valid. 

We record the dependence of these operators on the defining forms. For
$d\geq1$, choose Bombieri--Weyl orthonormal bases of $S_{d-1}$ and $S_d$.
With respect to these bases, the matrix of $\mu_{\alpha,d}$ depends linearly
on the coefficients of $\alpha$. Since multiplication by a nonzero linear
form is injective, the Hermitian operator
$\mu_{\alpha,d}^{*}\mu_{\alpha,d}$ is positive definite. Its inverse therefore
depends real-analytically on any local unit lift of $[\alpha]$. It follows
from \eqref{eq:multiplication-projector} and \eqref{eq:line-residual} that
$\Pi_{\alpha,d}$ and $\rho_{\alpha,d}$ also vary real-analytically with such
a lift. For $d=0$, the same conclusion follows directly from the conventions
above.

Using the unit representatives fixed above, put
\[
q_{\mathcal{A}}
=
\frac{Q(\mathcal{A})}{\|Q(\mathcal{A})\|_n}.
\]
The remaining phase choices do not affect the scalar quantities used below.
Indeed, if $\alpha_i$ is replaced by
\[
\alpha_i'=\zeta_i\alpha_i,
\qquad
|\zeta_i|=1,
\]
then
\[
\Pi_{\alpha_i',d}=\Pi_{\alpha_i,d},
\qquad
\rho_{\alpha_i',d}=\zeta_i\rho_{\alpha_i,d},
\qquad
q_{\mathcal{A}}'=
\left(\prod_{i=1}^n\zeta_i\right)q_{\mathcal{A}}.
\]
Consequently, the residual norms and
\[
\bigl|\langle B(u,v),q_{\mathcal{A}}\rangle_n\bigr|^2
\]
are independent of the chosen unit representatives. A permutation of the
lines merely permutes the blocks of $L_{\mathcal{A},d}$ and therefore also
leaves their direct-sum norm unchanged.

Fix a penalty strength $\lambda>0$ and an exponent $0<\beta<1$. Let $\mathbb S(E_d)=\{u\in E_d:\|u\|=1\}$. For
$(u,v)\in\mathbb S(E_{d_1})\times\mathbb S(E_{d_2})$, define the joint
logarithmic residual
\begin{equation}
\label{eq:joint-residual}
R_{\mathcal{A}}(u,v)
:=\|L_{\mathcal{A},d_1}u\|^2
+\|L_{\mathcal{A},d_2}v\|^2
\end{equation}
and
\begin{equation}
\label{eq:penalised-alignment}
\Gamma_{\lambda,\beta,\mathcal{A}}(u,v)
:=
\begin{cases}
\displaystyle
\frac{|\langle B(u,v),q_{\mathcal{A}}\rangle_n|^2}
{\|B(u,v)\|_n^2+\lambda R_{\mathcal{A}}(u,v)^\beta},
&\text{if the denominator is positive},\\[1.2em]
0,&\text{if the denominator is zero}.
\end{cases}
\end{equation}
The two cases are exhaustive because $\lambda>0$, $\beta>0$, and both terms in the denominator are nonnegative. In the second case, both the numerator and denominator vanish, so the value zero is an explicit definition at the otherwise undefined $0/0$ base locus.

\begin{definition}[Penalised Saito functional]
\label{def:saito-functional}
The penalised Saito functional associated with the prescribed degree pair
$(d_1,d_2)$ is
\begin{equation}
\label{eq:saito-loss}
\mathfrak S_{\lambda,\beta}(\mathcal{A};d_1,d_2)
=1-
\sup_{\substack{u\in\mathbb S(E_{d_1})\\
 v\in\mathbb S(E_{d_2})}}
\Gamma_{\lambda,\beta,\mathcal{A}}(u,v).
\end{equation}
\end{definition}

The numerator in \eqref{eq:penalised-alignment} measures the alignment of the candidate determinant with the unit Saito direction $q_{\mathcal{A}}$, while the second denominator term penalises failure of logarithmicity. Restricting to unit spheres eliminates reciprocal rescaling of the two inputs. The subquadratic power $\beta<1$ is essential near a common zero of the determinant and the logarithmic residual: for an arrangement that is not free with the selected exponent pair, it makes the penalty dominate the squared determinant and closes the base-locus escape.

\subsubsection{Range, zero locus, and semicontinuity}

Let
\[
\mathcal U_n=
\bigl\{([\alpha_1],\ldots,[\alpha_n])\in(\mathbb P(S_1))^n:
[\alpha_i]\neq[\alpha_j]\ \text{for}\ i\neq j\bigr\}
\]
be the ordered reduced configuration space of $n$-tuples of distinct projective lines, equipped with the Euclidean topology inherited from $(\mathbb P(S_1))^n$.

\begin{theorem}
\label{thm:penalised-saito}
Let $\mathcal{A}$ be a reduced arrangement of $n$ lines and let $d_1,d_2 \in \mathbb{Z}_{\geq0}$, satisfying $d_1+d_2=n-1$. For every $\lambda>0$ and $0<\beta<1$:
\begin{enumerate}[label=(\arabic*)]
\item $0\leq\mathfrak S_{\lambda,\beta}(\mathcal{A};d_1,d_2)\leq1$, and the supremum in its definition is attained;
\item $\mathfrak S_{\lambda,\beta}(\mathcal{A};d_1,d_2)=0$ if and only if $\mathcal{A}$ is free with exponents $(1,d_1,d_2)$;
\item if $\mathcal{A}$ is not free with the prescribed exponent pair, then
\[
0<\mathfrak S_{\lambda,\beta}(\mathcal{A};d_1,d_2)<1;
\]
\item for fixed $(d_1,d_2)$, the functional is upper semicontinuous on $\mathcal U_n$ and is continuous at every arrangement that is free with the prescribed exponent pair;
\item it is nondecreasing in $\lambda$, and
\[
\lim_{\lambda\to\infty}
\mathfrak S_{\lambda,\beta}(\mathcal{A};d_1,d_2)
=
\begin{cases}
0,&\mathcal{A}\text{ is free with exponents }(1,d_1,d_2),\\
1,&\text{otherwise}.
\end{cases}
\]
\end{enumerate}
\end{theorem}

\begin{proof}[Proof.]
\leavevmode\par\medskip
\begin{itemize}
\item \emph{Part (1).}
Since $q_{\mathcal{A}}$ is a unit vector, Cauchy--Schwarz gives
\[
|\langle B(u,v),q_{\mathcal{A}}\rangle_n|^2
\leq\|B(u,v)\|_n^2,
\]
so $0\leq\Gamma_{\lambda,\beta,\mathcal{A}}\leq1$ when the denominator is
 positive; in the other case $\Gamma_{\lambda,\beta, \mathcal{A}}=0$ by definition.

If $\mathcal{A}$ is free with exponents $(1,d_1,d_2)$, Saito's criterion provides $u_0\in D(\mathcal{A})_{d_1}$ and $v_0\in D(\mathcal{A})_{d_2}$ such that $B(u_0,v_0)=cQ({\mathcal{A}})$ with $c\neq0$. After normalising $u_0$ and $v_0$, $\|L_{\mathcal{A},d_1}u_0\|^2=\|L_{\mathcal{A},d_2}v_0\|^2=0$, and $\Gamma_{\lambda,\beta,\mathcal{A}}(u_0,v_0)=1$. Thus the supremum is attained.

Now suppose that $\mathcal{A}$ is not free with the prescribed pair. Write
\[
L_j=L_{\mathcal{A},d_j},
\qquad
K_j=\ker L_j=D(\mathcal{A})_{d_j}.
\]
Equation~\eqref{eq:kernel-determinant-zero} gives $B(K_1,K_2)=0$.

For unit vectors $u\in E_{d_1}$ and $v\in E_{d_2}$, write their orthogonal decompositions
\[
u=u_K+u_\perp,
\qquad
v=v_K+v_\perp,
\]
where $u_K\in K_1$, $u_\perp\in K_1^\perp$, $v_K\in K_2$, and $v_\perp\in K_2^\perp$. The restriction
\[
L_j|_{K_j^\perp}:K_j^\perp\longrightarrow \operatorname{im}L_j
\]
is injective: an element of its kernel belongs to both $K_j$ and $K_j^\perp$, and is therefore zero. It is also onto $\operatorname{im}L_j$, since $L_jw=L_jw_\perp$ for the orthogonal decomposition $w=w_K+w_\perp$. Because the spaces are finite-dimensional, this isomorphism has a bounded inverse. Equivalently, when $K_j^\perp\neq\{0\}$, compactness of its unit sphere gives
\[
\sigma_j:=
\min_{\substack{w\in K_j^\perp\\ \|w\|=1}}
\|L_jw\|>0.
\]
Thus, with $c_j=\sigma_j^{-1}$,
\[
\|u_\perp\|\leq c_1\|L_1u_\perp\|
   =c_1\|L_1u\|,
\qquad
\|v_\perp\|\leq c_2\|L_2v_\perp\|
   =c_2\|L_2v\|.
\]
If $K_j^\perp=\{0\}$, set $c_j=0$; then the corresponding estimate is trivial.

Let
\[
M=\|B\|_{\mathrm{bil}}
:=
\sup_{\|a\|=\|b\|=1}\|B(a,b)\|_n.
\]
This number is finite because $B$ is bilinear between finite-dimensional normed spaces. Since $B(u_K,v_K)=0$ by \eqref{eq:kernel-determinant-zero}, bilinearity gives
\[
B(u,v)
=
B(u_K,v_\perp)
+B(u_\perp,v_K)
+B(u_\perp,v_\perp).
\]
Orthogonal projection is norm-decreasing, so all four components have norm at most one. Consequently,
\begin{align*}
\|B(u,v)\|_n
&\leq
M\bigl(
\|u_K\|\|v_\perp\|
+\|u_\perp\|\|v_K\|
+\|u_\perp\|\|v_\perp\|
\bigr)\\
&\leq
\frac{3M}{2}
\bigl(\|u_\perp\|+\|v_\perp\|\bigr)\\
&\leq
\frac{3M}{2}
\bigl(c_1\|L_1u\|+c_2\|L_2v\|\bigr).
\end{align*}
Here we used $ab\leq(a+b)/2$ for $0\leq a,b\leq1$. Squaring and applying Cauchy--Schwarz to the final sum gives
\begin{equation}
\label{eq:B-residual-bound}
\|B(u,v)\|_n^2
\leq
C_{\mathcal{A};d_1,d_2}
\bigl(\|L_1u\|^2+\|L_2v\|^2\bigr)
=
C_{\mathcal{A};d_1,d_2}R_{\mathcal{A}}(u,v),
\end{equation}
where, for example, one may take
\[
C_{\mathcal{A};d_1,d_2}
=
\frac{9M^2}{4}(c_1^2+c_2^2).
\]

If $R_{\mathcal{A}}(u,v)=0$, then $u\in K_1$ and $v\in K_2$, so \eqref{eq:kernel-determinant-zero} gives $B(u,v)=0$; by definition, $\Gamma_{\lambda,\beta,\mathcal{A}}(u,v)=0$. If $R_{\mathcal{A}}(u,v)>0$, then $q_{\mathcal{A}}$ being a unit vector, Cauchy--Schwarz and \eqref{eq:B-residual-bound} give
\begin{align}
0\leq\Gamma_{\lambda,\beta,\mathcal{A}}(u,v)
&=
\frac{|\langle B(u,v),q_{\mathcal{A}}\rangle_n|^2}
{\|B(u,v)\|_n^2+\lambda R_{\mathcal{A}}(u,v)^\beta}
\nonumber\\
&\leq
\frac{\|B(u,v)\|_n^2}
{\lambda R_{\mathcal{A}}(u,v)^\beta}
\nonumber\\
&\leq
\frac{C_{\mathcal{A};d_1,d_2}}{\lambda}
R_{\mathcal{A}}(u,v)^{1-\beta}.
\label{eq:Gamma-base-bound}
\end{align}
Since $0<\beta<1$, the last expression tends to zero as $R_{\mathcal{A}}(u,v)\to0$. Hence, for this fixed arrangement outside the target free locus, the extension of $\Gamma$ by zero is continuous at its base locus. Away from that locus the quotient is already continuous. Therefore $\Gamma$ is continuous on the compact product of unit spheres and attains its maximum. Together with the prescribed-pair free case treated above, this proves part~(1).

\item \emph{Part (2).}
The free implication was proved above. Conversely, suppose that $\mathcal{A}$
is not free with the prescribed pair. If
$R_{\mathcal{A}}(u,v)=0$, then $u$ and $v$ are logarithmic, so
\eqref{eq:kernel-determinant-zero} gives $B(u,v)=0$ and hence, by definition,
$\Gamma_{\lambda,\beta,\mathcal{A}}(u,v)=0$. If
$R_{\mathcal{A}}(u,v)>0$, then
\[
|\langle B(u,v),q_{\mathcal{A}}\rangle_n|^2
\leq \|B(u,v)\|_n^2
<
\|B(u,v)\|_n^2+\lambda R_{\mathcal{A}}(u,v)^\beta,
\]
and therefore
$\Gamma_{\lambda,\beta,\mathcal{A}}(u,v)<1$. Thus $\Gamma<1$ for every
pair $(u,v)$. Since part~(1) shows that its maximum is attained, this maximum
is strictly smaller than one. Consequently,
\[
\mathfrak S_{\lambda,\beta}(\mathcal{A};d_1,d_2)>0.
\]

\item \emph{Part (3).}
We now prove that $\mathfrak S_{\lambda,\beta}(\mathcal{A};d_1,d_2)<1$, or equivalently that $\Gamma_{\lambda,\beta,\mathcal{A}}$ is not identically zero. The linear span of $B(E_{d_1},E_{d_2})$ is all of $S_n$. Indeed, let $\mathbf{m}$ be a monomial of degree $n$, choose a variable $x_j$ dividing $\mathbf{m}$, and write
\[
\mathbf{m}/x_j=\mathbf{pq},
\qquad
\deg \mathbf{p}=d_1,\qquad
\deg \mathbf{q}=d_2.
\]
If $x_k$ and $x_\ell$ are the other two coordinate variables, take
\[
u=\mathbf{p}\partial_{x_k},
\qquad
v=\mathbf{q}\partial_{x_\ell}.
\]
Then $B(u,v)=\pm \mathbf{m}$. Since the monomials form an orthogonal basis of $S_n$
and $q_{\mathcal{A}}\neq0$, some such pair satisfies
\[
\langle B(u,v),q_{\mathcal{A}}\rangle_n\neq0.
\]
After normalising $u$ and $v$, the numerator in \eqref{eq:penalised-alignment} remains nonzero, and therefore
\[
\Gamma_{\lambda,\beta,\mathcal{A}}(u,v)>0.
\]
Hence $\sup\Gamma_{\lambda,\beta,\mathcal{A}}>0$ and
\[
\mathfrak S_{\lambda,\beta}(\mathcal{A};d_1,d_2)<1.
\]


\item \emph{Part (4).}
Fix $\mathcal{A}_0\in\mathcal U_n$. In a neighbourhood of $\mathcal{A}_0$,
choose local continuous unit lifts of the defining line forms. By the
preceding discussion, for every fixed pair $(u,v)$ the quantities
\[
R_{\mathcal{A}}(u,v)
\qquad\text{and}\qquad
\bigl|\langle B(u,v),q_{\mathcal{A}}\rangle_n\bigr|^2
\]
depend continuously on $\mathcal{A}$ and are independent of the chosen local
lifts.

Where the denominator in \eqref{eq:penalised-alignment} is positive,
$\Gamma_{\lambda,\beta,\mathcal{A}}(u,v)$ is therefore continuous in
$\mathcal{A}$. If the denominator vanishes at $\mathcal{A}_0$, then
$\Gamma_{\lambda,\beta,\mathcal{A}_0}(u,v)=0$ by definition. Since $\Gamma$ is
nonnegative, this gives
\[
\Gamma_{\lambda,\beta,\mathcal{A}_0}(u,v)
\leq
\liminf_{\mathcal{A}\to\mathcal{A}_0}
\Gamma_{\lambda,\beta,\mathcal{A}}(u,v),
\]
so $\mathcal{A}\mapsto
\Gamma_{\lambda,\beta,\mathcal{A}}(u,v)$ is lower semicontinuous for every
fixed $(u,v)$.

A supremum of lower-semicontinuous functions is lower semicontinuous. Hence
\[
\mathcal{A}
\longmapsto
\sup_{\substack{u\in\mathbb S(E_{d_1})\\
v\in\mathbb S(E_{d_2})}}
\Gamma_{\lambda,\beta,\mathcal{A}}(u,v)
\]
is lower semicontinuous, and therefore
\[
\mathcal{A}
\longmapsto
\mathfrak S_{\lambda,\beta}(\mathcal{A};d_1,d_2)
\]
is upper semicontinuous on $\mathcal U_n$.

Finally, suppose that $\mathcal{A}_0$ is free with the prescribed exponent
pair. By part~(2),
\[
\mathfrak S_{\lambda,\beta}(\mathcal{A}_0;d_1,d_2)=0.
\]
For every sequence $\mathcal{A}_k\to\mathcal{A}_0$, nonnegativity and upper
semicontinuity give
\[
0
\leq
\liminf_{k\to\infty}
\mathfrak S_{\lambda,\beta}(\mathcal{A}_k;d_1,d_2)
\leq
\limsup_{k\to\infty}
\mathfrak S_{\lambda,\beta}(\mathcal{A}_k;d_1,d_2)
\leq0.
\]
Thus
\[
\mathfrak S_{\lambda,\beta}(\mathcal{A}_k;d_1,d_2)
\longrightarrow0
=
\mathfrak S_{\lambda,\beta}(\mathcal{A}_0;d_1,d_2),
\]
which proves continuity at every arrangement that is free with the prescribed
exponent pair.

\item \emph{Part (5).}
For every fixed pair $(u,v)$, the quantity
$\Gamma_{\lambda,\beta,\mathcal{A}}(u,v)$ is nonincreasing in $\lambda$,
because increasing $\lambda$ can only increase the denominator in
\eqref{eq:penalised-alignment}. Hence
\[
\lambda
\longmapsto
\sup_{\substack{u\in\mathbb S(E_{d_1})\\
v\in\mathbb S(E_{d_2})}}
\Gamma_{\lambda,\beta,\mathcal{A}}(u,v)
\]
is nonincreasing, and therefore
$\mathfrak S_{\lambda,\beta}(\mathcal{A};d_1,d_2)$ is nondecreasing in
$\lambda$.

If $\mathcal{A}$ is free with exponents $(1,d_1,d_2)$, then part~(2) gives
\[
\mathfrak S_{\lambda,\beta}(\mathcal{A};d_1,d_2)=0
\]
for every $\lambda>0$.

Suppose now that $\mathcal{A}$ is not free with the prescribed exponent pair,
and abbreviate
\[
C=C_{\mathcal{A};d_1,d_2}.
\]
Since $R_{\mathcal{A}}$ is continuous on the compact product of unit spheres,
the number
\[
R_{\max}
:=
\max_{\substack{u\in\mathbb S(E_{d_1})\\
v\in\mathbb S(E_{d_2})}}
R_{\mathcal{A}}(u,v)
\]
is finite.

Fix a unit pair $(u,v)$ and write
\[
R=R_{\mathcal{A}}(u,v).
\]
If $R=0$, then \eqref{eq:B-residual-bound} gives $B(u,v)=0$, and hence
$\Gamma_{\lambda,\beta,\mathcal{A}}(u,v)=0$ by definition. If $R>0$, then
Cauchy--Schwarz and \eqref{eq:B-residual-bound} give
\begin{align*}
\Gamma_{\lambda,\beta,\mathcal{A}}(u,v)
&\leq
\frac{\|B(u,v)\|_n^2}
{\|B(u,v)\|_n^2+\lambda R^\beta}\\
&\leq
\frac{CR}{CR+\lambda R^\beta}\\
&=
\frac{CR^{1-\beta}}
{CR^{1-\beta}+\lambda}\\
&\leq
\frac{CR_{\max}^{1-\beta}}
{CR_{\max}^{1-\beta}+\lambda}.
\end{align*}
For the second inequality, we used
$\|B(u,v)\|_n^2\leq CR$ and the fact that
\[
x\longmapsto\frac{x}{x+\lambda R^\beta}
\]
is increasing for $x\geq0$. The final inequality follows from
$0<\beta<1$ and $R\leq R_{\max}$.

The resulting bound is independent of $(u,v)$. Consequently,
\[
0
\leq
\sup_{\substack{u\in\mathbb S(E_{d_1})\\
v\in\mathbb S(E_{d_2})}}
\Gamma_{\lambda,\beta,\mathcal{A}}(u,v)
\leq
\frac{CR_{\max}^{1-\beta}}
{CR_{\max}^{1-\beta}+\lambda}
\longrightarrow0
\]
as $\lambda\to\infty$. Therefore
\[
\mathfrak S_{\lambda,\beta}(\mathcal{A};d_1,d_2)
=
1-
\sup_{\substack{u\in\mathbb S(E_{d_1})\\
v\in\mathbb S(E_{d_2})}}
\Gamma_{\lambda,\beta,\mathcal{A}}(u,v)
\longrightarrow1.
\]

\end{itemize}
\end{proof}


\begin{corollary}
\label{corollary:envelope-functional}
 If $\mathscr D$ is a nonempty subset of $\{(d_1,d_2) \in \mathbb{Z}_{\geq0}^2 : d_1+d_2=n-1\}$, the envelope
\begin{equation}
\label{eq:pair-envelope}
\mathfrak S^{\mathscr D}_{\lambda,\beta}(\mathcal{A})
=\min_{(d_1,d_2)\in\mathscr D}
\mathfrak S_{\lambda,\beta}(\mathcal{A};d_1,d_2)
\end{equation}
is upper semicontinuous and vanishes precisely when $\mathcal{A}$ is free with one of the admissible pairs. 
\end{corollary}


\begin{remark}[Normalisation and invariance]
The zero locus in Theorem~\ref{thm:penalised-saito} is intrinsic, but positive values depend on the Hermitian metric, $\lambda$, and $\beta$. The Bombieri--Weyl construction is invariant under unitary coordinate changes, multiplication of each unit defining form by a scalar of modulus one, and line permutation. It is not invariant under arbitrary $\operatorname{PGL}_3(\mathbb C)$ changes and should therefore be described as a bounded coordinate-normalised residual energy, not as a projective invariant or a metric distance to the free locus. For a real arrangement, optimisation over complex coefficient vectors can only decrease the energy relative to optimisation over the real unit spheres; the two versions nevertheless have the same zero locus.
\end{remark}
\begin{remark}[Classical families]
Theorem~\ref{thm:penalised-saito} applies uniformly to the standard free
families. In particular, pencils, supersolvable arrangements, and reflection
arrangements have functional value zero at their exponent pair. For every
arrangement that is not free with the prescribed pair, the functional lies
strictly between zero and one. The searches described below exclude pencils
and near-pencils. For Ziegler-type pairs and, more generally, for distinct
realisations of the same intersection lattice, positive values need not
agree, since the functional depends on the realisation and the chosen
Hermitian normalisation; only its vanishing has the exact freeness
interpretation.
\end{remark}
\begin{remark}[Numerical approximation]
The supremum in \eqref{eq:saito-loss} is a nonconvex optimisation problem. If a finite-restart optimiser returns a feasible value $\widehat\Gamma$, then, in exact arithmetic,
\[
\widehat\Gamma\leq\sup\Gamma,
\qquad
\widehat{\mathfrak S}:=1-\widehat\Gamma
\geq\mathfrak S_{\lambda,\beta}.
\]
Thus the computed energy is an upper approximation to the ideal functional. A large computed value does not prove non-freeness, while a small value only selects a candidate for exact verification. Every reported discovery is certified over its exact coefficient field by finding
$u\in D(\mathcal{A})_{d_1}$ and $v\in D(\mathcal{A})_{d_2}$ such that $B(u,v)=cQ({\mathcal{A}})$ with $c\neq0$.
\end{remark}

\subsection{Energy-guided fixed-cardinality search}
\label{subsec:sequential-construction}

The search is carried out at a prescribed cardinality $n\geq2$. Its basic move replaces one line by another:
\begin{equation}
\label{eq:line-replacement}
\mathcal{A}'
=
(\mathcal{A}\setminus\{L_-\})\cup\{L_+\}.
\end{equation}
Separate campaigns are run over the exact fields
\begin{equation}
\label{eq:search-fields}
\mathscr K=
\left\{
\mathbb{Q},
\mathbb{Q}(\sqrt2),
\mathbb{Q}(\sqrt3),
\mathbb{Q}(\sqrt5),
\mathbb{Q}(\sqrt{-1}),
\mathbb{Q}(\sqrt{-3})
\right\}.
\end{equation}

For each $K\in\mathscr K$, replacement lines are generated exactly over $K$ from a finite, state-dependent proposal set. This set combines a fixed pool of lines of bounded coefficient height with lines passing through pairs of intersection points of $\mathcal{A}\setminus\{L_-\}$. Thus the pool available after selecting $L_-$ depends on the current arrangement. Projective proportionality, incidences and $b_2$ are computed exactly over $K$. For the numerical evaluation of the penalised Saito functional, we fix an embedding $K\hookrightarrow\mathbb C$. This choice may affect positive numerical values of the functional, but it does not affect the existence of an exact Saito certificate over $K$.

A proposed state is called \emph{admissible} if it consists of $n$ distinct lines, is essential, and has maximal intersection multiplicity at most $n-2$. These conditions exclude pencils and near-pencils. In campaigns directed towards a prescribed multiplicity gap, admissibility may additionally impose a fixed ceiling
\[
\max_{X\in \mathcal{L}_2(\mathcal{A})}m_X\leq M_0.
\]
A proposal violating any of these hard conditions is rejected before numerical evaluation.

Fix a target pair
\[
d_1,d_2\in\mathbb{Z}_{\geq0},
\qquad
d_1+d_2=n-1,
\]
and write $b_2^*=b_2^*(n;d_1,d_2)$ as in \eqref{eq:b2-target}. The proposal distribution is biased towards the necessary $b_2$-shell. Candidate replacements are stratified according to their exactly computed effect on $b_2$: priority is given to proposals satisfying
\[
b_2(\mathcal{A}')=b_2^*,
\]
while a secondary tier allows
\[
|b_2(\mathcal{A}')-b_2^*|\leq2.
\]
This proposal-level use of $b_2$ is distinct from the auxiliary discrepancy term in \eqref{eq:production-energy}: the former shapes the candidates presented to the search engines, whereas the latter enters only the selection or acceptance rules of the energy-based engines.

We use simulated annealing~\cite{Kirkpatrick1983}, greedy search~\cite{Cuntz2021}, a random-walk baseline~\cite{Solis1981}, and MAP--Elites~\cite{mouret2015illuminatingsearchspacesmapping,Cully2018-nm}. These methods share the same replacement mechanism, numerical evaluator and exact-certification procedure, but use the numerical values differently in their transition and retention rules.

For the energy-based engines, the ideal composite energy is
\begin{equation}
\label{eq:ideal-search-energy}
\mathcal E(\mathcal{A};d_1,d_2)
:=
\mathfrak S_{1,\,3/4}(\mathcal{A};d_1,d_2)
+w_{b_2}
\frac{|b_2(\mathcal{A})-b_2^*|}{b_2^*},
\qquad
w_{b_2}=0.05.
\end{equation}
Numerically, the supremum defining the first term is approximated by a finite multistart maximisation of $\Gamma_{1,3/4,\mathcal{A}}$. If $\widehat{\Gamma}_{\mathcal{A};d_1,d_2}$ denotes the largest feasible value found, set
\[
\widehat{\mathfrak S}^{\,\mathrm{raw}}_{1,3/4}
(\mathcal{A};d_1,d_2)
=
1-\widehat{\Gamma}_{\mathcal{A};d_1,d_2}.
\]
Here \emph{raw} means that no further rescaling or calibration is applied. The reported evaluations use real coefficient vectors for $\mathbb Q$ and the real quadratic fields, and complex coefficient vectors for the imaginary quadratic fields.

The corresponding computed composite energy is
\begin{equation}
\label{eq:production-energy}
\widehat{\mathcal E}(\mathcal{A};d_1,d_2)
:=
\widehat{\mathfrak S}^{\,\mathrm{raw}}_{1,3/4}
(\mathcal{A};d_1,d_2)
+0.05
\frac{|b_2(\mathcal{A})-b_2^*|}{b_2^*}.
\end{equation}
The second term depends only on the intersection lattice. It gives a weak additional preference for the necessary $b_2$-shell, while the slack proposal tier still permits nearby off-shell states.

Greedy search and simulated annealing use $\widehat{\mathcal E}$ in their selection or acceptance rules. The random walk accepts admissible proposals without consulting the energy and serves as a baseline. MAP--Elites organises arrangements into cells determined by combinatorial descriptors and retains elites according to the order of magnitude of the raw penalised Saito loss, with additional preferences for non-supersolvable representatives, smaller raw loss, and smaller coordinate height. 

In exact arithmetic, the value of $\Gamma$ at every feasible numerical iterate satisfies
\[
\widehat{\Gamma}_{\mathcal{A};d_1,d_2}
\leq
\sup_{\substack{u\in\mathbb S(E_{d_1})\\
 v\in\mathbb S(E_{d_2})}}
\Gamma_{1,3/4,\mathcal{A}}(u,v),
\]
and hence
\[
\widehat{\mathfrak S}^{\,\mathrm{raw}}_{1,3/4}
(\mathcal{A};d_1,d_2)
\geq
\mathfrak S_{1,3/4}(\mathcal{A};d_1,d_2).
\]
For the floating-point computation, this inequality is understood up to numerical evaluation error. Objective values that violate the theoretical range beyond the prescribed numerical tolerance are re-evaluated using more accurate arithmetic; unresolved numerical failures are discarded rather than interpreted as loss values. Thus a small computed loss identifies a promising candidate, whereas a large computed loss does not prove non-freeness, since the numerical maximisation may fail to locate the global supremum.

\begin{algorithm}[t]
\caption{Line-replacement search at fixed $(n,d_1,d_2)$}
\label{alg:general-search}
\begin{algorithmic}[1]
\Require an exact field $K\in\mathscr K$, an initial collection
$\mathscr X$ of admissible $n$-line arrangements, a target pair
$d_1+d_2=n-1$, a search engine, and a search budget $T$
\Ensure a collection $\mathscr A$ of exactly certified free arrangements
\State $\mathscr A\gets\varnothing$
\While{the search budget has not been exhausted}
  \State Select a parent $\mathcal{A}\in\mathscr X$ and propose one or more
  replacements
  \[
  \mathcal{A}'
  =
  (\mathcal{A}\setminus\{L_-\})\cup\{L_+\}
  \]
  according to the chosen engine and the $b_2$-biased proposal rule.
  \For{each proposed arrangement $\mathcal{A}'$}
  \If{$\mathcal{A}'$ is admissible}
  \State Evaluate
  $\widehat{\mathfrak S}^{\,\mathrm{raw}}_{1,3/4}
  (\mathcal{A}';d_1,d_2)$ and $b_2(\mathcal{A}')$.
  \State If required by the engine, compute
  $\widehat{\mathcal E}(\mathcal{A}';d_1,d_2)$.
  \If{$\widehat{\mathfrak S}^{\,\mathrm{raw}}_{1,3/4}
  (\mathcal{A}';d_1,d_2)<10^{-6}$ and its intersection-lattice
  fingerprint has not previously been attempted in the
current campaign unit}
 \If{$\mathcal{A}'$ satisfies Saito's criterion with degrees
 $(d_1,d_2)$ exactly over $K$}
 \State Add $\mathcal{A}'$ and its certificate to
 $\mathscr A$.
 \EndIf
  \EndIf
  \EndIf
  \EndFor
  \State Update the current state or archive according to the chosen engine.
\EndWhile
\State \Return $\mathscr A$.
\end{algorithmic}
\end{algorithm}

The common structure of the search is summarised in Algorithm~\ref{alg:general-search}. Here $T$ denotes a generic search budget, which may be specified in terms of wall time, iterations, generations, or numerical evaluations, depending on the engine.


Each individual search stage preserves the cardinality $n$. To reach larger cardinalities, some campaigns are organised into cascades: a certified arrangement is extended by one line and used to initialise a new fixed-cardinality search at $n+1$, after which the replacement procedure resumes. The penalised Saito functional and the combinatorial proposal bias therefore guide the exploration, while exact Saito certification provides the final correctness guarantee.

\section{Certified free arrangements with large multiplicity gaps}
\label{sec:experimental_results}

\subsection{The multiplicity gap and Terao's conjecture}

Let $\mathcal{A}$ be a free arrangement of $n$ lines with exponents
\[
(1,d_1,d_2),
\qquad
d_1\leq d_2,
\]
and let
\[
m(\mathcal{A})
=
\max_{X\in \mathcal{L}_2(\mathcal{A})}m_X
\]
denote its maximal intersection multiplicity. Following Dimca--K\"uhne--Pokora~\cite{Dimca2026freelinearrangementslow}, we consider the integer
\begin{equation}
\label{eq:multiplicity-gap}
\epsilon(\mathcal{A})
=
d_1-m(\mathcal{A}).
\end{equation}
We call $\mathcal{A}$ a \emph{large-gap arrangement} when
$\epsilon(\mathcal{A})\geq2$.

For the comparison with Terao's conjecture, suppose that
\[
\epsilon(\mathcal{A})>0.
\]
By definition,
\[
d_1=m(\mathcal{A})+\epsilon(\mathcal{A}).
\]
Moreover, since $d_1\leq d_2$ and $d_1+d_2=n-1$, we have
\begin{align*}
n-m(\mathcal{A})-d_1
&=
d_2+1-m(\mathcal{A})\\
&\geq
d_1+1-m(\mathcal{A})\\
&=
\epsilon(\mathcal{A})+1
>0.
\end{align*}
Consequently,
\[
d_1
=
m(\mathcal{A})+\epsilon(\mathcal{A})
<
n-m(\mathcal{A}),
\]
so the hypotheses of the comparison below are satisfied. In particular,
these hypotheses hold for every large-gap arrangement.

Suppose now that $\overline{\mathcal{A}}$ is another realisation of the same
intersection lattice as $\mathcal{A}$, but that $\overline{\mathcal{A}}$ is not
free. Since maximal multiplicity is determined by the intersection lattice,
\[
m(\overline{\mathcal{A}})=m(\mathcal{A}).
\]
Let
\[
\overline d_1
=
\operatorname{mdr}(Q_{\overline{\mathcal{A}}})
:=
\min\left\{
s\geq0:
\operatorname{AR}(Q_{\overline{\mathcal{A}}})_s\neq0
\right\}
\]
be the minimal degree of a Jacobian relation of
$Q_{\overline{\mathcal{A}}}$. The comparison in
\cite{Dimca2026freelinearrangementslow} then gives
\begin{equation}
\label{eq:terao-gap-window}
m(\mathcal{A})
\leq
\overline d_1
<
m(\mathcal{A})+\epsilon(\mathcal{A}).
\end{equation}
Thus $\overline d_1$ would have to belong to
\[
\bigl\{
m(\mathcal{A}),\,
m(\mathcal{A})+1,\ldots,\,
m(\mathcal{A})+\epsilon(\mathcal{A})-1
\bigr\}.
\]
Because $\epsilon(\mathcal{A})$ is a positive integer, this set contains
exactly $\epsilon(\mathcal{A})$ possible values. This is the precise sense in
which a larger value of $\epsilon$ leaves more ``space'' for a possible
counterexample to Terao's conjecture: it enlarges the arithmetic range of
values not excluded for
$\operatorname{mdr}(Q_{\overline{\mathcal{A}}})$. It does not refer to the
dimension of the realisation space, nor does it imply that a counterexample
is likely to exist. The arrangements presented below supply only the free
side of this comparison. To obtain a counterexample to Terao's
conjecture~\cite{Terao80}, one would still need to exhibit a non-free
realisation having the same intersection lattice as one of them.

Large-gap arrangements are automatically non-supersolvable. Indeed, if a rank-three supersolvable arrangement has a modular point of multiplicity $r$, then its exponents are
\[
(1,r-1,n-r)
\]
up to ordering; see \cite[Theorem~4.58]{Orlik1992}. Consequently,
\[
d_1\leq r-1\leq m(\mathcal{A})-1,
\]
and hence $\epsilon(\mathcal{A})<0$. Therefore $\epsilon(\mathcal{A})\geq2$ already excludes supersolvable families, although it does not by itself imply that an arrangement is not inductively, recursively, or divisionally free.

More generally, Dimca--K\"uhne--Pokora prove
\cite[Proposition~3.1]{Dimca2026freelinearrangementslow} that a free
arrangement satisfying
\[
d_1=m+\epsilon,
\qquad
m=m(\mathcal{A}),
\qquad
\epsilon\geq0,
\]
must satisfy
\begin{equation}
\label{eq:DKP-bounds}
2m+2\epsilon+1
\leq n
\leq
\frac{m(m+2+\epsilon)}{2}.
\end{equation}
We shall call an arrangement \emph{lower-bound extremal} when equality holds
on the left of \eqref{eq:DKP-bounds}. Since
$d_1+d_2=n-1$, this equality gives
\[
n-1=2m+2\epsilon=2d_1,
\]
and hence
\[
d_1=d_2=m+\epsilon.
\]

\subsection{Certified arrangements found by the search}

Every arrangement counted in this section was verified over its exact coefficient field using Saito's criterion. More precisely, for each retained arrangement we computed homogeneous logarithmic derivations
\[
\theta_1\in D(\mathcal{A})_{d_1},
\qquad
\theta_2\in D(\mathcal{A})_{d_2},
\]
and verified coefficientwise that
\begin{equation}
\label{eq:results-exact-saito}
B(\theta_1,\theta_2)
=
cQ(\mathcal{A}),
\qquad
c\neq0.
\end{equation}
The penalised Saito loss and the $b_2$-shell condition were used only to select promising candidates; neither a small numerical loss nor the factorisation of the characteristic polynomial was accepted as a proof of
freeness.

The complete collection was subsequently audited anew from its exact data. For each of the $6{,}146$ retained representatives, we reconstructed the arrangement over its declared coefficient field, checked the field consistency of the certificate, re-established the logarithmicity of both derivations, verified \eqref{eq:results-exact-saito} symbolically, and recomputed the coloured line--point incidence graph. All $6{,}146$ representatives passed these checks, with no failures.

The certified discoveries yield $6{,}146$ distinct Weisfeiler--Leman fingerprints\footnote{We encode each intersection lattice by its coloured line--point incidence graph and associate with it a Weisfeiler--Leman fingerprint~\cite{Kiefer2020}. The fingerprint is invariant under graph isomorphism but is not a complete isomorphism invariant: distinct fingerprints certify non-isomorphic incidence graphs, whereas equal fingerprints need not certify isomorphism. Consequently, unless complete graph-isomorphism comparisons are also performed, the number of distinct fingerprints gives a lower bound for the number of lattice-isomorphism types occurring among the discoveries.}, and include arrangements with cardinalities up to $n=28$. One representative of each fingerprint is retained. Thus at least $6{,}146$ lattice-isomorphism types occur among the certified discoveries, although this is not claimed to be a complete
classification.

The coefficient fields of the retained representatives are distributed as
follows:

\[
\begin{array}{c|r|r|r|r|r|r}
\text{Coefficient field}
&
~\mathbb{Q}~
&
~\mathbb{Q}(\sqrt{-3})~
&
~\mathbb{Q}(\sqrt3)~
&
~\mathbb{Q}(\sqrt{-1})~
&
~\mathbb{Q}(\sqrt5)~
&
~\mathbb{Q}(\sqrt2)~
\\ \hline
\#\text{ representatives}
&
5{,}808&160&90&82&5&1
\end{array}
\]

These are fields over which the retained representatives are defined; they are not claimed to be their minimal fields of definition. 

Among the retained representatives, $3{,}012$ satisfy $\epsilon(\mathcal{A})\geq2$; their distribution is shown in Table~\ref{tab:epsilon-distribution}. Since exactly one representative is retained per Weisfeiler--Leman fingerprint, and distinct fingerprints certify non-isomorphic coloured incidence graphs, each count in Table~\ref{tab:epsilon-distribution} is the exact number of pairwise non-isomorphic intersection lattices among the retained representatives in that row. As a consistency check on the fingerprint implementation, all retained large-gap representatives were additionally compared by exact isomorphism tests of their coloured line--point incidence graphs within each fixed $(\epsilon,n,d_1,d_2)$ class; as expected, no coincidences were found. Among all certified discoveries, these counts remain lower bounds on the number of lattice-isomorphism types, since a discovery discarded for sharing a fingerprint with a retained representative need not be isomorphic to it.

\begin{table}[ht]
\centering
\caption{Distribution of the $3{,}012$ retained certified large-gap representatives. Because exactly one representative is retained per distinct Weisfeiler--Leman fingerprint, each row counts pairwise non-isomorphic intersection lattices among the retained representatives. As an independent consistency check, pairwise graph-isomorphism tests within each fixed $(\epsilon,n,d_1,d_2)$ class found no coincidences.}

\label{tab:epsilon-distribution}
\begin{tabular}{c|r|c}
\hline
$\epsilon$
&
count
&
observed cardinalities $n$
\\
\hline
$2$ & $1{,}132$ & $14$--$22$ and $24$--$25$ \\
$3$ & $1{,}068$ & $17$--$26$ \\
$4$ & $522$ & $19$--$27$ \\
$5$ & $215$ & $23$--$28$ \\
$6$ & $71$ & $25$--$28$ \\
$7$ & $4$  & $27$\\
\hline
\end{tabular}
\end{table}

These counts describe the output of a targeted search and should not be interpreted as frequencies among all line arrangements. Moreover, many of the larger arrangements were obtained through successive extensions of a small number of initial families. The counts therefore measure the combinatorial types found by the search, rather than their prevalence or independence in realisation space. Nevertheless, the results show that the targeted search can produce many exactly certified examples in the large-gap regime.

\subsection{Arrangements attaining the lower bound}

The certified collection contains lower-bound extremal representatives for $21$ distinct parameter pairs $(m,\epsilon)$ at $n=2m+2\epsilon+1$. Every such arrangement necessarily has balanced exponents. 

Among the largest-gap extremal parameter pairs found by the search are
\[
\begin{array}{c|c|c|c|r}
~n~&~m~&~\epsilon~&~(d_1,d_2)~&~\text{retained count}~\\
\hline
17&5&3&(8,8)&23\\
19&5&4&(9,9)&5\\
23&6&5&(11,11)&43\\
25&6&6&(12,12)&38\\
27&6&7&(13,13)&4
\end{array}
\]
For all five rows, the displayed numbers are exact: the retained representatives within each class were checked pairwise by explicit graph isomorphism of their coloured incidence graphs, with no coincidences found, so each count equals the number of pairwise non-isomorphic intersection lattices among the retained representatives.

In particular, the collection contains extremal arrangements with $m=6$ attaining the lower bound consecutively for $\epsilon=5,6,7$, at cardinalities $n=23,25,27$, respectively. In each row, the displayed cardinality is the smallest permitted by \eqref{eq:DKP-bounds} for the specified pair $(m,\epsilon)$. This does not assert global minimality among all arrangements having the
same value of $\epsilon$.

\begin{example}[A rational lower-bound extremal arrangement with multiplicity gap seven]
\label{ex:epsilon-seven-rational}

Let $[x:y:z]$ be homogeneous coordinates on $\mathbb P^2$, and define
\[
\mathcal{A}_{27}=\{L_1,\ldots,L_{27}\}
\subset\mathbb P^2_{\mathbb Q},
\]
where the lines are as follows\footnote{The complete certificate for this arrangement is recorded in the accompanying
database under the Weisfeiler--Leman fingerprint
\texttt{1a1488047e120a539063022fff9d160f}.}:
\begin{center}
\small
\renewcommand{\arraystretch}{1.12}
\begin{tabular}{r@{\;}l@{\qquad}r@{\;}l}
$L_1=$ & $V(x-5y+5z)$
&
$L_{15}=$ & $V(x-35y+4z)$
\\
$L_2=$ & $V(x-7y+4z)$
&
$L_{16}=$ & $V(x-6y+\frac92z)$
\\
$L_3=$ & $V(x-9y+3z)$
&
$L_{17}=$ & $V(x-15y+\frac73z)$
\\
$L_4=$ & $V(x+5y+\frac{13}{2}z)$
&
$L_{18}=$ & $V(x-3y+6z)$
\\
$L_5=$ & $V(x-25y+9z)$
&
$L_{19}=$ & $V(x-10y+\frac{29}{6}z)$
\\
$L_6=$ & $V(y+\frac12z)$
&
$L_{20}=$ & $V(x+\frac52y+\frac{21}{4}z)$
\\
$L_7=$ & $V(x-5y+\frac32z)$
&
$L_{21}=$ & $V(x-5y+\frac{53}{12}z)$
\\
$L_8=$ & $V(x+4z)$
&
$L_{22}=$ & $V(x-\frac{45}{7}y+\frac{23}{7}z)$
\\
$L_9=$ & $V(x-\frac{25}{7}y+\frac{33}{7}z)$
&
$L_{23}=$ & $V(x-10y+\frac{27}{8}z)$
\\
$L_{10}=$ & $V(x-\frac{25}{3}y+\frac{17}{3}z)$
&
$L_{24}=$ & $V(x-5y+\frac{22}{3}z)$
\\
$L_{11}=$ & $V(x-45y-z)$
&
$L_{25}=$ & $V(x-15y+\frac{21}{4}z)$
\\
$L_{12}=$ & $V(x-5y+4z)$
&
$L_{26}=$ & $V(x-\frac52y+\frac{11}{4}z)$
\\
$L_{13}=$ & $V(x-\frac{15}{2}y+\frac{37}{8}z)$
&
$L_{27}=$ & $V(x-\frac{35}{3}y+4z)$
\\
$L_{14}=$ & $V(x-\frac{35}{4}y+4z)$
&
&
\end{tabular}
\end{center}
\begin{figure}[H]
  \centering
  \includegraphics[width=0.6\linewidth]{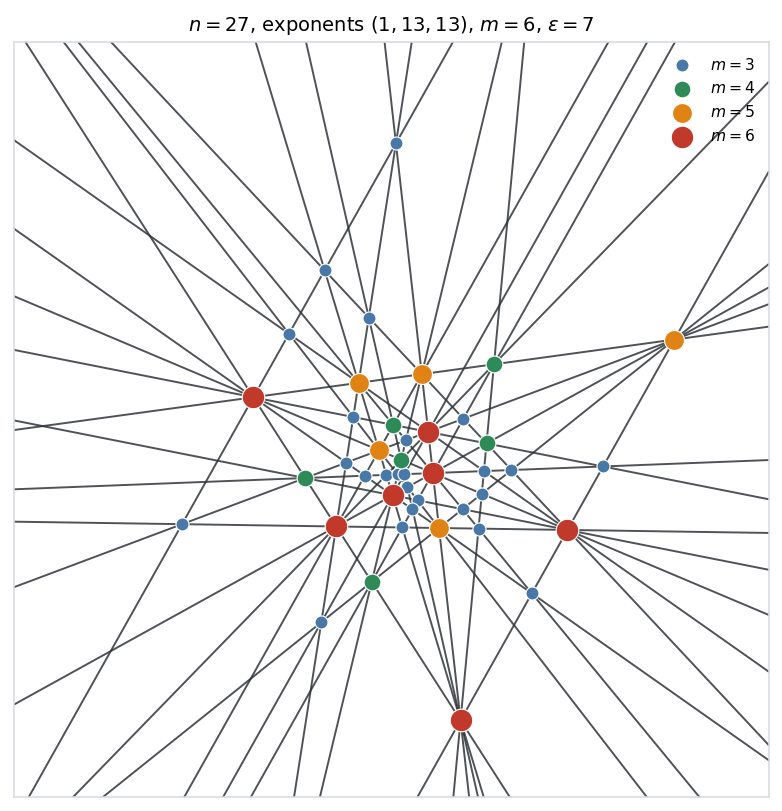}
  \caption{An affine drawing of the projective arrangement $\mathcal{A}_{27}$ in the chart $z=1$. This arrangement satisfies $m(\mathcal{A}_{27})=6$ and $\epsilon(\mathcal{A}_{27})=7$.}
  \label{fig:A27}
\end{figure}

Writing $L_i=V(\alpha_i)$, set
\[
Q({\mathcal{A}_{27}})=\prod_{i=1}^{27}\alpha_i.
\]

An exact incidence computation gives
\begin{equation}
\label{eq:epsilon-seven-profile}
(t_2,t_3,t_4,t_5,t_6)
=
(57,29,7,6,7),
\end{equation}
with $t_r=0$ for $r\geq7$. In particular,
\[
m(\mathcal{A}_{27})=6.
\]
As a consistency check, every pair of lines is accounted for:
\[
\sum_{r\geq2}\binom{r}{2}t_r
=
57+3\cdot29+6\cdot7+10\cdot6+15\cdot7
=
351
=
\binom{27}{2}.
\]
Moreover,
\[
\begin{aligned}
b_2(\mathcal{A}_{27})
&=
\sum_{r\geq2}(r-1)t_r\\
&=
57+2\cdot29+3\cdot7+4\cdot6+5\cdot7\\
&=
195
=
(27-1)+13^2.
\end{aligned}
\]
Consequently,
\[
\chi_{\mathcal{A}_{27}}(t)
=
(t-1)(t^2-26t+169)
=
(t-1)(t-13)^2.
\]

These combinatorial identities are necessary consistency checks, but freeness is established by an exact Saito certificate. More precisely,
there exist
\[
\theta_1,\theta_2\in D(\mathcal{A}_{27})_{13}
\]
such that
\[
B(\theta_1,\theta_2)
=
\frac{128625}{875446}\,Q({\mathcal{A}_{27}}).
\]
Since the scalar is nonzero, Theorem~\ref{thm:Saito-criterion} shows that $\mathcal{A}_{27}$ is free with exponents $(1,13,13)$. Its multiplicity gap is
\[
\epsilon(\mathcal{A}_{27})
=
13-m(\mathcal{A}_{27})
=
13-6
=
7.
\]
Furthermore,
\[
27=2\cdot6+2\cdot7+1,
\]
so $\mathcal{A}_{27}$ attains the lower bound in
\eqref{eq:DKP-bounds}. It is therefore lower-bound extremal, with balanced
exponents. It is also non-supersolvable, since a supersolvable free line
arrangement must satisfy
\[
d_1\leq m(\mathcal{A})-1.
\]

For any hypothetical non-free realisation $\overline{\mathcal{A}}_{27}$ having the same intersection lattice as $\mathcal{A}_{27}$,
\eqref{eq:terao-gap-window} would require
\[
6
\leq
\operatorname{mdr}\bigl(Q({\overline{\mathcal{A}}_{27}})\bigr)
<
13,
\]
or equivalently
\[
\operatorname{mdr}\bigl(Q({\overline{\mathcal{A}}_{27}})\bigr)
\in
\{6,7,8,9,10,11,12\}.
\]
This seven-value interval is the widest such window represented in the present collection.

\end{example}

The examples found here therefore provide a collection of exactly certified, non-supersolvable free arrangements in the arithmetic regime highlighted by Dimca--K\"uhne--Pokora~\cite{Dimca2026freelinearrangementslow}. They supply only the free side of the comparison relevant to Terao's conjecture~\cite{Terao80}. To obtain a counterexample, one would still need to exhibit a non-free realisation having the same intersection lattice as one of them and prove its non-freeness exactly.

\section{Conclusion and future perspectives}
\label{sec:conclusion}
We introduced the penalised Saito functional
$\mathfrak S_{\lambda,\beta}(\mathcal{A};d_1,d_2)$ as a bounded search energy
on the space of reduced line arrangements, for each prescribed pair
$d_1+d_2=n-1$. The construction optimises the Saito determinant over
derivations of the prescribed degrees while penalising their failure to be
logarithmic, measured by canonical residual operators associated with the
Bombieri--Weyl Hermitian structure. Theorem~\ref{thm:penalised-saito} shows
that the resulting functional takes values in $[0,1]$, vanishes precisely on
arrangements free with exponents $(1,d_1,d_2)$, and is upper semicontinuous
on the reduced configuration space. Its positive values depend on the chosen
Hermitian normalisation and on the parameters $\lambda$ and $\beta$.
Accordingly, they should be interpreted as heuristic search energies, rather
than as projective invariants or intrinsic metric distances to the free
locus.

We incorporated numerical approximations of this functional into fixed-cardinality line-replacement searches, together with the auxiliary $b_2$ discrepancy in \eqref{eq:production-energy}. The different search engines share the same evaluator and exact certification pipeline. Numerical energies are used only to rank and explore candidates, while every reported free arrangement is verified over its exact coefficient field through the Saito identity \eqref{eq:results-exact-saito}.

At the audited snapshot, the certified collection contains representatives with $6{,}146$ distinct Weisfeiler--Leman fingerprints and cardinalities up to $n=28$. Of these, $3{,}012$ satisfy
\[
\epsilon(\mathcal{A})=d_1-m(\mathcal{A})\geq2
\]
and are therefore non-supersolvable. Lower-bound-extremal arrangements occur for $21$ distinct parameter pairs $(m,\epsilon)$, including
\[
(n,m,\epsilon)
=
(17,5,3),\quad
(19,5,4),\quad
(23,6,5),\quad
(25,6,6),\quad
(27,6,7).
\]
In particular, the collection reaches multiplicity gap $\epsilon=7$. These figures describe the output of a targeted search: they are neither frequency estimates nor a complete classification. The total fingerprint count is a conservative lower bound for the number of lattice-isomorphism types occurring among the discoveries.

The resulting large-gap arrangements provide exactly certified, non-supersolvable test cases in the regime highlighted by Dimca--K\"uhne--Pokora~\cite{Dimca2026freelinearrangementslow}. They supply only the free side of the comparison relevant to Terao's conjecture. Producing a counterexample would require a second, non-free realisation of one of the same intersection lattices, together with an exact proof of non-freeness. Future work will expand the exactly certified dataset. Progress in this direction depends both on available computational resources and on methodological choices, including the coefficient fields and proposal pools explored, and the effectiveness of the numerical optimiser and search strategies. A larger collection should make it possible to search systematically for patterns in the combinatorial data of free arrangements and to identify lattice properties associated with freeness, exponent distributions, and large multiplicity gaps. Particular attention will be given to the realisation spaces of the most promising lattices, with the aim of finding distinct realisations exhibiting different freeness behaviour and thereby producing a possible counterexample to Terao's conjecture.

\medskip
\bmhead{Acknowledgements}

The author thanks Henrique N. Sá Earp, Marcos B. Jardim, Abbas N. Nejad, Simone Marchesi, Felipe C. F. Monteiro, and Edward Hirst for helpful discussions during the early stages of this project. The author gratefully acknowledges the valuable feedback received during the \emph{DANGER: Data, Numbers and Geometry} workshop, held in Banff in April 2026. Special thanks are due to Ilkyoo Choi for several insightful and stimulating discussions. The author also thanks the two anonymous reviewers for their careful reading and constructive feedback. The author acknowledges support from the São Paulo Research Foundation (FAPESP), grant 2022/09891-4. This research utilised computational resources provided by the ``\emph{Centro Nacional de Processamento de Alto Desempenho em São Paulo} (CENAPAD-SP)''. The accompanying codebase was developed with the assistance of Claude Code (Fable 5), an AI coding tool developed by Anthropic.

\section*{Data and code availability}
Related data and code are publicly available at: \url{https://github.com/TomasSilva/FreeLineArrangements}










\bibliography{ref}

\end{document}

%% file: init.tex
\renewcommand{\det}{\mathop\mathrm{det}\nolimits}

\renewcommand{\epsilon}{\varepsilon}

\def\<{\mathopen{}\left<}
\def\>{\right>\mathclose{}}
\def\({\mathopen{}\left(}
\def\){\right)\mathclose{}}

\usepackage{multicol, color}

\definecolor{gold}{rgb}{0.85,.66,0}
\definecolor{cherry}{rgb}{0.9,.1,.2}
\definecolor{burgundy}{rgb}{0.8,.2,.2}
\definecolor{orangered}{rgb}{0.85,.3,0}
\definecolor{orange}{rgb}{0.85,.4,0}
\definecolor{olive}{rgb}{.45,.4,0}
\definecolor{lime}{rgb}{.6,.9,0}
\definecolor{green}{rgb}{.2,.7,0}
\definecolor{grey}{rgb}{.4,.4,.2}
\definecolor{brown}{rgb}{.4,.3,.1}

\newtheorem{theorem}{Theorem}

\newtheorem{corollary}{Corollary}
\newtheorem{lemma}{Lemma}

\theoremstyle{remark}
\newtheorem{remark}{Remark}

\theoremstyle{definition}
\newtheorem{definition}{Definition}
\newtheorem{example}{Example}